\documentclass[11pt]{article}
\usepackage{amsfonts,latexsym,rawfonts,amsmath,amssymb,amsthm}
\usepackage{lscape}
\usepackage[cm]{fullpage}
\usepackage{amscd, float,times,rotating}
\usepackage{pb-diagram}
\usepackage{hyperref}
\usepackage{amsthm}
\usepackage{enumerate}
\usepackage[latin1]{inputenc}
\usepackage[T1]{fontenc}
\usepackage{mathrsfs}
\usepackage{amscd}
\usepackage{amssymb}
\usepackage[english]{babel}
\usepackage{graphicx}

\usepackage{geometry}
\newtheorem{theorem}{Theorem}[section]
\newtheorem{lemma}[theorem]{Lemma}

\newtheorem{proposition}[theorem]{Proposition}

\theoremstyle{definition}
\newtheorem{definition}[theorem]{Definition}

\numberwithin{equation}{section}

\begin{document}

\title{$C^{2,\alpha}$ regularities and estimates for nonlinear elliptic and parabolic equations in geometry\footnote{This article has been accepted for publication on Calculus of Variations and Partial Differential Equations. This is accepted version. The final publication is available at Springer via http://dx.doi.org/10.1007/s00526-015-0948-5}}
\author{Jianchun Chu}
\date{}
\maketitle

\begin{abstract}
We give sharp $C^{2,\alpha}$ estimates for solutions of some fully nonlinear elliptic and parabolic equations in complex geometry and almost complex geometry, assuming a bound on the Laplacian of the solution. We also prove the analogous results to complex Monge-Amp\`{e}re equations with conical singularities. As an application, we obtain a local estimate for Calabi-Yau equation in almost complex geometry. We also improve the $C^{2,\alpha}$ regularities and estimates for viscosity solutions to some uniformly elliptic and parabolic equations.  All our results are optimal regarding the H\"{o}lder exponent.
\end{abstract}

\section{Introduction}
Fully nonlinear elliptic and parabolic equations have natural connections with problems in geometry. And Schauder estimates for solutions to these equations with $C^{\alpha}$ right hand side are of considerable interest. Recently, Tosatti-Wang-Weinkove-Yang established $C^{2,\alpha}$ estimates for solutions of some nonlinear elliptic equations in complex geometry, assuming a bound on the Laplacian of the solution.

In this paper, we improve their result. On the basis of their work, we can lift $\alpha$ to $\gamma_{0}=\min(\alpha_{0},\beta_{0})$ (we explain $\alpha_{0}$ and $\beta_{0}$ later), which is the optimal H\"{o}lder exponent. For the reader's convenience, most of our notations are the same with \cite{TWWY}.

Let $(M,J,\omega)$ be a compact Hermitian manifold. Suppose $u\in C^{2}(M)$ is a real-valued function which satisfies
\begin{equation*}
\|u\|_{L^{\infty}(M)}\leq K~~and~~\Delta u\leq K,
\end{equation*}
where
\begin{equation*}
\Delta u=\frac{n\sqrt{-1}\partial\bar{\partial} u\wedge\omega^{n-1}}{\omega^{n}}.
\end{equation*}
We assume that $\psi\in C^{\alpha_{0}}(M)$ and $\chi$ is a real $(1,1)$ form with coefficients in $C^{\beta_{0}}(M)$, where $\alpha_{0},\beta_{0}\in(0,1)$. We point out that $\psi$ and $\chi$ may depend on $u$.

The partial differential equations discussed in this paper include complex Monge-Amp\`{e}re equation, complex Hessian equations, complex $\sigma_{n}/\sigma_{k}$ equations, Monge-Amp\`{e}re equation for $(n-1)$-plurisubharmonic equations, $(n-1)$-plurisubharmonic version of the complex Hessian and $\sigma_{n}/\sigma_{k}$ equations and almost complex versions of all of the above. Let us recall some results about these equations (for more details, see Section 1 of \cite{TWWY}, which is a good survey of these equations). Since all these equations (including $\psi$ and $\chi$) can be found in Section 1 of \cite{TWWY}, we do not list them in the following.

We discuss the complex Monge-Amp\`{e}re equation first. When $\chi$ is a fixed K\"{a}hler metric, the existence of solution was proved in the famous work of Yau \cite{Ya}. In \cite{Ya}, Yau used Calabi's estimate \cite{Ca} to establish the $C^{2,\alpha}$ estimate, which depends on third derivatives of $\psi$. In \cite{Si}, Siu used the Evans-Krylov approach to get the $C^{2,\alpha}$ estimate which depends on second derivatives of $\psi$ (also \cite{Tr}). In \cite{Ti83}, Tian presented a new proof of the $C^{2,\alpha}$ estimate (for real and complex Monge-Amp\`{e}re equations), which not only weakens the regularity assumptions on $\psi$ but also can be applied to more general nonlinear elliptic systems. His $C^{2,\alpha}$ estimate depends on the H\"{o}lder norm of $\psi$ and lower bound of $\Delta\psi$. Recently, Tian \cite{Ti13} extended his method to the conic case. In \cite{DZZ}, Dinew-Zhang-Zhang use a new method to establish the $C^{2,\alpha}$ estimate depending on the H\"{o}lder bound of $\psi$ and the bound for the real Hessian of $u$. And their estimate is optimal according to the H\"{o}lder exponent. A $C^{2,\alpha}$ estimate depending on the H\"{o}lder bound of $\psi$ and the bound of $\Delta u$ was given by Wang \cite{Wa}, and also Chen-Wang \cite{CW1402}.

When $\chi$ is just a Hermitian metric, the existence of solution was proved by Cherrier \cite{Ch} for $n=2$ (and with other hypotheses when $n>2$) and by Tosatti-Weinkove \cite{TW2} for any dimensions. In \cite{Ch}, Cherrier established the $C^{2,\alpha}$ estimate by using Calabi's estimate. In \cite{GL}, Guan-Li got the estimate for the real Hessian of $u$ by maximum principle first, then they proved the $C^{2,\alpha}$ estimate by Evans-Krylov theory. In \cite{TW1}, Tosatti-Weinkove obtained the $C^{2,\alpha}$ estimate via a similar approach given in \cite{Si}. All these estimates depend on at least second derivatives of $\psi$.

For the complex Hessian equations with $\chi=\omega$ K\"{a}hler, the existence of solution was proved by Dinew-Ko{\l}odziej \cite{DK12}. The $C^{2,\alpha}$ estimate was established by applying Evans-Krylov theory, which depends on second derivatives of $\psi$.

The complex $\sigma_{n}/\sigma_{k}$ equations can be regarded as generalizations of the complex Hessian equations. When $k=n-1$ and $\chi$ is a fixed K\"{a}hler metric, Song-Weinkove \cite{SW} obtained the necessary and sufficient conditions for the existence of solution via a parabolic method (see \cite{We04, We06}). Fang-Lai-Ma \cite{FLM} got the analogous result for general $k$. When $\chi$ is just a Hermitian metric, The complex $\sigma_{n}/\sigma_{k}$ equations were solved by Sun \cite{Su} (see also \cite{GS, Li1}). Note that the $C^{2,\alpha}$ estimate in \cite{Su} depends on second derivatives of $\psi$.

For the Monge-Amp\`{e}re equation for $(n-1)$-plurisubharmonic equations, Tosatti-Weinkove \cite{TW3} proved the existence of solution when $\chi$ is a fixed Hermitian metric and $\omega$ is a K\"{a}hler metric. And they solved the case of $\omega$ Hermitian in \cite{TW4} recently. The $C^{2,\alpha}$ estimates in these cases were established in \cite{TW3, TW4} by Evans-Krylov approach, which depends on second derivatives of $\psi$.

As far as we know, the $(n-1)$-plurisubharmonic version of the complex Hessian and $\sigma_{n}/\sigma_{k}$ equations have not been investigated up to now. And there is almost not existing result in the literature.

For almost complex versions of all of the above equations, Harvey-Lawson \cite{HL11} and Pli\'{s} \cite{Pl} solved the Dirichlet problem of the almost complex Monge-Amp\`{e}re equation. In \cite{De}, Delano\"{e} studied a related but different equation in the almost complex case. Calabi-Yau equation can be transformed into the almost complex Monge-Amp\`{e}re equation locally. For this equation, Evans-Krylov results were proved in Tosatti-Weinkove-Yau \cite{TWY} and Weinkove \cite{We07}.

In this paper, we prove
\begin{theorem}\label{Main Theorem}
Suppose that $u\in C^{2}(M)$ satisfies one of the equations above. Then we have $u\in C^{2,\gamma_{0}}(M)$ and the following estimate
\begin{equation*}
\|u\|_{C^{2,\gamma_{0}}(M)}\leq C,
\end{equation*}
where $\gamma_{0}=\min(\alpha_{0},\beta_{0})$ and $C$ depends only on $n$, $\alpha_{0}$, $\beta_{0}$, $(M,J,\omega)$, $K$, $\|\psi\|_{C^{\alpha_{0}}(M)}$ and $\|\chi\|_{C^{\beta_{0}}(M)}$.
\end{theorem}

In \cite{TWWY}, Tosatti-Wang-Weinkove-Yang proved a similar result. They got $C^{2,\alpha}$ regularity and estimate of the solution $u$, where $\alpha$ depends on the same background data in Theorem \ref{Main Theorem}. However, $\alpha$ is probably very small. Our result is optimal regarding the H\"{o}lder exponent. That is, under the assumptions of Theorem \ref{Main Theorem}, for any $\tilde{\alpha}\in(\gamma_{0},1)$, there exist $\chi$ and $\psi$ such that the equations above do not admit any $C^{2,\tilde{\alpha}}$ solutions.

Next, we consider the complex Monge-Amp\`{e}re equations with conical singularities. Let $(B_{0}(1),\omega_{\beta})$ be the singular space, where $0<\beta<1$ and $\omega_{\beta}$ is the model cone metric
\begin{equation*}
\omega_{\beta}=\sqrt{-1}\frac{\beta^{2}}{|z|^{2-2\beta}}dz\wedge d\bar{z}+\sqrt{-1}\sum_{k=2}^{n}dz_{k}\wedge d\bar{z}_{k},
\end{equation*}
where $(z,z_{2},\cdot\cdot\cdot,z_{n})$ is the standard coordinates of $B_{0}(1)\subset\mathbb{C}^{n}$.

Our result is as follows.
\begin{theorem}\label{New Conical Theorem}
Let $\alpha_{0}\in(0,\min(\frac{1}{\beta}-1,1))$ be a constant. Suppose $\phi$ is a plurisubharmonic function in $C^{2,\gamma,\beta}(B_{0}(1))$ for any $\gamma\in(0,\alpha_{0})$. Suppose
\begin{equation*}
\det\phi_{i\bar{j}}=e^{f}~~and~~\frac{1}{K}\omega_{\beta}\leq\sqrt{-1}\partial\bar{\partial}\phi\leq K\omega_{\beta}~~over~B_{0}(1)\backslash\{0\}\times\mathbb{C}^{n-1}
\end{equation*}
for some $K>1$, where $f\in C^{\alpha_{0},\beta}(B_{0}(1))$. Then we have $\phi\in C^{2,\alpha_{0},\beta}(B_{0}(\frac{1}{2}))$ and the following estimate
\begin{equation*}
[\sqrt{-1}\partial\bar{\partial}\phi]_{C^{\alpha_{0},\beta}(B_{0}(\frac{1}{2}))}\leq C,
\end{equation*}
where $C$ depends only on $n$, $\alpha_{0}$, $\beta$, $K$, $\|\phi\|_{L^{\infty}(B_{0}(1))}$ and $\|f\|_{C^{\alpha_{0},\beta}(B_{0}(1))}$.
\end{theorem}
I will explain the notations $C^{\alpha,\beta}$ and $C^{2,\alpha,\beta}$ in Section 3.

In \cite{CW1402}, Chen-Wang proved a analogous estimate. But our result is optimal regarding the H\"{o}lder exponent, which is similar to Theorem \ref{Main Theorem}.

The paper is organized as follows. In Section 2 and Section 3 we prove Theorem \ref{Main Theorem} and Theorem \ref{New Conical Theorem} respectively. In Section 4 and Section 5 we describe how to generalize our results on more general elliptic and parabolic equations respectively. In Section 6 we give an application of Theorem \ref{Main Theorem} on Calabi-Yau equation. Finally, in Section 7 we recall some basic results which are crucial to our proof of Theorem \ref{Main Theorem}.

\noindent\textbf{Acknowledgments} ~~The author would like to thank his advisor Gang Tian for constant encouragement and several useful comments on an earlier version of this paper. The author would also like to thank Yuanqi Wang for some helpful discussions about complex Monge-Amp\`{e}re equations with conical singularities. The author would also like to thank Valentino Tosatti and Jingang Xiong for many helpful conversations. The author would also like to thank CSC for supporting the author visiting Northwestern University.

\section{The proof of Theorem \ref{Main Theorem}}
In this section, we prove Theorem \ref{Main Theorem}. As we can see, Theorem 1.2 in \cite{TWWY} is very important in the proof of Theorem \ref{Main Theorem}. For the reader's convenience, we state it first.

We shall use the following notations. If $M=(m_{ij})$ is a real matrix, we write $\|M\|=(\sum_{i,j}m_{ij}^{2})^{\frac{1}{2}}$. If $M=(m_{ij})$ is a complex matrix, we write $\|M\|=(\sum_{i,j}|m_{ij}|^{2})^{\frac{1}{2}}$.

Let $Sym(2n)$ be the space of symmetric $2n\times2n$ matrices with real entries. We consider equation of the form
\begin{equation*}
F\left(S(x)+T(D^{2}u(x),x),x\right)=f(x)~~in~B_{0}(1),
\end{equation*}
where $f\in C^{\alpha_{0}}(B_{0}(1))$ and
\begin{equation*}
\begin{split}
&F:Sym(2n)\times B_{0}(1)\rightarrow \mathbb{R};\\
&S:B_{0}(1)\rightarrow Sym(2n);\\
&T:Sym(2n)\times B_{0}(1)\rightarrow Sym(2n).
\end{split}
\end{equation*}
We assume that there exists a compact convex set $\mathcal{E}\subset Sym(2n)$, positive constants $\lambda$, $\Lambda$, $K$ and $\beta_{0}\in(0,1)$ such that the following hold.
\begin{enumerate}[\textbf{H}1:]
    \item  $F$ is of class $C^{1}$ in $U\times B_{0}(1)$ where $U$ is a neighborhood of $\mathcal{E}$ and
           \begin{enumerate}[(1)]
           \item $F$ is uniformly elliptic in $\mathcal{E}$:
                 \begin{equation*}
                 \lambda|\xi|^{2}\leq F_{ij}(M,x)\xi^{i}\xi^{j}\leq\Lambda|\xi|^{2},
                 \end{equation*}
                 for all $M\in\mathcal{E}$, $x\in B_{0}(1)$ and $\xi\in\mathbb{R}^{2n}$, where $F_{ij}(M,x)=\frac{\partial F}{\partial m_{ij}}(M,x)$ and $m_{ij}$ are the components of the matrix $M$.
           \item $F$ is concave in $\mathcal{E}$:
                 \begin{equation*}
                 F(\frac{A+B}{2},x)\geq\frac{1}{2}F(A,x)+\frac{1}{2}F(B,x),
                 \end{equation*}
                 for all $A,B\in\mathcal{E}$ and $x\in B_{0}(1)$.
           \item $F$ has the following uniform H\"{o}lder bound in $x$:
                 \begin{equation*}
                 |F(N,x)-F(N,y)|\leq K|x-y|^{\beta_{0}}~~and~~|F(N,0)|\leq K,
                 \end{equation*}
                 for all $N\in\mathcal{E}$ and $x,y\in B_{0}(1)$.
           \end{enumerate}

    \item  The map $T:Sym(2n)\times B_{0}(1)\rightarrow Sym(2n)$ satisfies the following conditions:
           \begin{enumerate}[(1)]
           \item For all $x,y\in B_{0}(1)$ and all $N\in Sym(2n)$,
                 \begin{equation*}
                 \frac{\|T(N,x)-T(N,y)\|}{\|N\|+1}\leq K|x-y|^{\beta_{0}}.
                 \end{equation*}
           \item For each fixed $x\in B_{0}(1)$, the map $M\mapsto T(M,x)$ is linear on $Sym(2n)$. For convenience, we
                 assume $T_{ij}(M,x)=T_{ij,kl}(x)m_{kl}$.
           \item For all $P\geq0$ and $x\in B_{0}(1)$,
                 \begin{equation*}
                 T(P,x)\geq0~~and~~K^{-1}\|P\|\leq\|T(P,x)\|\leq K\|P\|.
                 \end{equation*}
           \end{enumerate}

    \item  $S:B_{0}(1)\rightarrow Sym(2n)$ has a uniform $C^{\beta_{0}}$ bound:
           \begin{equation*}
           \|S(x)-S(y)\|\leq K|x-y|^{\beta_{0}}~~and~~\|S(0)\|\leq K,
           \end{equation*}
           for all $x,y\in B_{0}(1)$.
\end{enumerate}

The following Theorem is Theorem 1.2 in \cite{TWWY}.
\begin{theorem}[Tosatti-Wang-Weinkove-Yang \cite{TWWY}]\label{TWWY PDE Theorem}
With the assumption above, suppose that $u\in C^{2}(B_{0}(1))$ solves
\begin{equation*}
F\left(S(x)+T(D^{2}u(x),x),x\right)=f(x)~~in~B_{0}(1),
\end{equation*}
and satisfies
\begin{equation*}
S(x)+T(D^{2}u(x),x)\in\mathcal{E},~~\forall~x\in B_{0}(1).
\end{equation*}
Then $u\in C^{2,\alpha}(B_{0}(\frac{1}{2}))$ and
\begin{equation*}
\|u\|_{C^{2,\alpha}(B_{0}(\frac{1}{2}))}\leq C,
\end{equation*}
where $\alpha$ and C depends only on $\alpha_{0}$, $K$, $n$, $\Lambda$, $\lambda$, $\beta_{0}$, $\|f\|_{C^{\alpha_{0}}(B_{0}(1))}$ and $\|u\|_{L^{\infty}(B_{0}(1))}$.
\end{theorem}

In order to prove Theorem \ref{Main Theorem}, we introduce the following PDE theorem.
\begin{theorem}\label{PDE Theorem}
With the assumption above, suppose that $u\in C^{2}(B_{0}(1))$ solves
\begin{equation*}
F\left(S(x)+T(D^{2}u(x),x),x\right)=f(x)~~in~B_{0}(1),
\end{equation*}
and satisfies
\begin{equation*}
S(x)+T(D^{2}u(x),x)\in\mathcal{E},~~\forall~x\in B_{0}(1).
\end{equation*}
Then $u\in C^{2,\gamma_{0}}(B_{0}(\frac{1}{2}))$ and
\begin{equation*}
\|u\|_{C^{2,\gamma_{0}}(B_{0}(\frac{1}{2}))}\leq C,
\end{equation*}
where $\gamma_{0}=\min(\alpha_{0},\beta_{0})$ and C depends only on $\alpha_{0}$, $K$, $n$, $\Lambda$, $\lambda$, $\beta_{0}$, $\|f\|_{C^{\alpha_{0}}(B_{0}(1))}$, $\|u\|_{L^{\infty}(B_{0}(1))}$ and the moduli of continuity of $F_{ij}$.
\end{theorem}

As we can see, Theorem \ref{PDE Theorem} and Theorem \ref{TWWY PDE Theorem} are similar. Indeed, their assumptions are the same. However, our results are different. In Theorem \ref{TWWY PDE Theorem}, Tosatti-Wang-Weinkove-Yang \cite{TWWY} proved the existence of $\alpha$. In our paper, we obtain the $C^{2,\gamma_{0}}$ regularity and estimate of the solution, which is optimal regarding the H\"{o}lder exponent.

In the Section 2 of \cite{TWWY}, Tosatti-Wang-Weinkove-Yang described how to apply Theorem \ref{TWWY PDE Theorem} to obtain the $C^{2,\alpha}$ regularity and estimate of the solution. Since Theorem \ref{PDE Theorem} and Theorem \ref{TWWY PDE Theorem} have the same assumptions, we can use Theorem \ref{PDE Theorem} to prove Theorem \ref{Main Theorem} by the same argument. Compared with Theorem \ref{TWWY PDE Theorem}, the constant $C$ in Theorem \ref{PDE Theorem} need to depend on the moduli of continuity of $F_{ij}$. But we still can use Theorem \ref{PDE Theorem} to prove Theorem \ref{Main Theorem} by the same method in \cite{TWWY}. Because the moduli of continuity of $F_{ij}$ depends only on $n$, $(M,J,\omega)$, $K$ and $\|\psi\|_{L^{\infty},M}$ in all equations we refer in Section 1 (see Section 2 of \cite{TWWY}). Therefore, it suffices to prove Theorem \ref{PDE Theorem}.

In fact, Theorem \ref{TWWY PDE Theorem} plays an important role in our argument. In order to control second derivatives of $u$, we need to use Theorem \ref{TWWY PDE Theorem} to get $C^{2,\alpha}$ estimate first (maybe $\alpha$ is very small), then we can lift $\alpha$ to $\gamma_{0}=\min(\alpha_{0},\beta_{0})$.

By Theorem \ref{TWWY PDE Theorem}, we have the estimate of $D^{2}u$, which implies
\begin{equation*}
-LI_{2n}\leq D^{2}u(x)\leq LI_{2n}~~\forall x\in B_{0}(\frac{3}{4}),
\end{equation*}
where $L$ depends only on $\alpha_{0}$, $K$, $n$, $\Lambda$, $\lambda$, $\beta_{0}$, $\|f\|_{C^{\alpha_{0}}(B_{0}(1))}$ and $\|u\|_{L^{\infty}(B_{0}(1))}$. Let $P[-L,L]$ be the space of matrices with eigenvalues lying in the interval $[-L,L]$. For any $H\in P[-L,L]\cap\mathcal{E}$ and $q\in B_{0}(1)$, we define
\begin{equation*}
\Psi_{H,q}(M,x)=F\left(S(x)+T(M,x),x\right)-a_{kl}(H,q)m_{kl},
\end{equation*}
where $a_{kl}(H,q)=F_{ij}\left(S(q)+T(H,q),q\right)T_{ij,kl}(q)$, then we have the following lemmas.

\begin{lemma}\label{uniformly elliptic}
Under the assumptions of Theorem \ref{PDE Theorem}, we have
\begin{equation*}
K^{-1}\lambda I_{2n}\leq\left(a_{ij}(H,q)\right)\leq \sqrt{2n}K\Lambda I_{2n},
\end{equation*}
for all $H\in\mathcal{E}$ and $q\in B_{0}(1)$.
\end{lemma}

\begin{proof}
For any $B=(b_{ij})\geq0$, we have $T(B,q)\geq0$ by \textbf{H}2 (3). By uniformly ellipicity, we obtain
\begin{equation*}
\lambda I_{2n}\leq F_{ij}\left(S(H)+T(H,q),q\right)\leq \Lambda I_{2n}.
\end{equation*}
It then follows that
\begin{equation*}
\lambda\|T(B,q)\|\leq F_{ij}\left(S(q)+T(H,q),q\right)T_{ij}(B,q) \leq\sqrt{2n}\Lambda\|T(B,q)\|,
\end{equation*}
which implies
\begin{equation*}
K^{-1}\lambda\|B\|\leq a_{kl}(H,q)b_{kl} \leq\sqrt{2n}K\Lambda\|B\|,
\end{equation*}
where we use \textbf{H}2 (3) again. Since $B$ is arbitrary, we complete the proof.
\end{proof}

\begin{lemma}\label{regularity proposition}
Under the assumptions of Theorem \ref{PDE Theorem}, we have $u\in C^{2,\gamma}(B_{0}(\frac{3}{4}))$ for any $\gamma\in(0,\gamma_{0})$.
\end{lemma}

\begin{proof}
For any unit vector $e$ and $h\in(0,\frac{1}{2})$, we have
\begin{equation}\label{regularity proposition equation1}
F\left(S(x+he)+T(D^{2}u(x+he),x+he),x+he\right)=f(x+he)~~in~B_{0}(\frac{1}{3}).
\end{equation}
It is clear that
\begin{equation}\label{regularity proposition equation2}
F\left(S(x)+T(D^{2}u(x),x),x\right)=f(x)~~in~B_{0}(\frac{1}{3}).
\end{equation}
Combining (\ref{regularity proposition equation1}) and (\ref{regularity proposition equation2}), we obtain
\begin{equation}\label{regularity proposition equation3}
\begin{split}
 &\tilde{a}_{ij}(x,h,e)T_{ij,kl}(x+he)\frac{u_{kl}(x+he)-u_{kl}(x)}{h^{\gamma_{0}}}\\
=&-\frac{F\left(S(x)+T(D^{2}u(x),x),x+he\right)-F\left(S(x)+T(D^{2}u(x),x),x\right)}{h^{\gamma_{0}}}\\
&+\frac{f(x+he)-f(x)}{h^{\gamma_{0}}}-\tilde{a}_{ij}(x,h,e)\frac{S_{ij}(x+he)-S_{ij}(x)}{h^{\gamma_{0}}}\\
&-\tilde{a}_{ij}(x,h,e)\frac{T_{ij}(D^{2}u(x),x+he)-T_{ij}(D^{2}u(x),x)}{h^{\gamma_{0}}},
\end{split}
\end{equation}
where
\begin{equation*}
\begin{split}
\tilde{a}_{ij}(x,h,e)=&\int_{0}^{1}F_{ij}\bigg(tS(x+he)+(1-t)S(x)+tT\left(D^{2}u(x+he),x+he\right)\\
                      &+(1-t)T\left(D^{2}u(x),x\right),x+he\bigg)dt.
\end{split}
\end{equation*}
We denote by $\tilde{\tilde{f}}(x,h,e)$ the right hand side in (\ref{regularity proposition equation3}). It then follows that
\begin{equation}\label{regularity proposition equation4}
\tilde{\tilde{a}}_{kl}(x,h,e)(D_{e}^{h,\gamma_{0}}u)_{kl}(x)=\tilde{\tilde{f}}(x,h,e)~~in~B_{0}(\frac{1}{3}),
\end{equation}
where
\begin{equation*}
\tilde{\tilde{a}}_{kl}(x,h,e)=\tilde{a}_{ij}(x,h,e)T_{ij,kl}(x+he)~~and~~D_{e}^{h,\gamma_{0}}u(x)=\frac{u(x+he)-u(x)}{h^{\gamma_{0}}}.
\end{equation*}
By the assumptions of $F$, $S$ and $T$, we find that (\ref{regularity proposition equation4}) is a uniformly elliptic partial differential equation. By $L^{p}$ estimates, we have
\begin{equation*}
\|D_{e}^{h,\gamma_{0}}u\|_{W^{2,p}(B_{0}(\frac{1}{4}))}\leq C,
\end{equation*}
for any $1<p<\infty$, where $C$ depends only on $n$, $p$, $\|D_{e}^{h,\gamma_{0}}u\|_{L^{p}(B_{0}(\frac{1}{3}))}$, $\|\tilde{\tilde{f}}\|_{L^{p}(B_{0}(\frac{1}{3}))}$, the positive lower and upper bounds on the eigenvalues of $(\tilde{\tilde{a}}_{kl})$ and the moduli of continuity of the coefficients $\tilde{\tilde{a}}_{kl}$. And we can check that all these data are independent of $h$ and $e$. Thus, the constant $C$ is independent of $h$ and $e$ ($C$ maybe depends on $u$). By Theorem \ref{Lp regular theorem}, we obtain $u\in C^{2,\gamma}(B_{0}(\frac{1}{6}))$ for any $\gamma\in(0,\gamma_{0})$. Then, the desired result follows from a simple covering argument.
\end{proof}

\begin{lemma}\label{matrix lemma}
For any $\varepsilon>0$, there exists a positive number $\delta$ depending only on $\varepsilon$, $n$, $\Lambda$, $K$, $L$, $\beta_{0}$ and the moduli of continuity of $F_{ij}$, with the following property. For any $H, N, \tilde{N}\in P[-L,L]\cap\mathcal{E}$ and $q, x, y\in B_{0}(\frac{3}{4})$, if
\begin{equation*}
\|N-H\|+\|\tilde{N}-H\|+|x-q|+|y-q|<\delta,
\end{equation*}
then we have
\begin{equation*}
|\Psi_{H,q}(N,x)-\Psi_{H,q}(\tilde{N},y)|\leq\varepsilon\|N-\tilde{N}\|+K(1+4n\Lambda+2n\Lambda L)|x-y|^{\beta_{0}}.
\end{equation*}
\end{lemma}

\begin{proof}
By the definition of $\Psi$, we compute
\begin{equation}\label{matrix lemma equation1}
\begin{split}
&|\Psi_{H,q}(N,x)-\Psi_{H,q}(N,y)|\\
=~&|F\left(S(x)+T(N,x),x\right)-F\left(S(y)+T(N,y),y\right)|\\
\leq~&|F\left(S(x)+T(N,x),x\right)-F\left(S(x)+T(N,x),y\right)|\\
+&|F\left(S(x)+T(N,x),y\right)-F\left(S(y)+T(N,y),y\right)|\\
\leq~&K|x-y|^{\beta_{0}}+2n\Lambda\|S(x)+T(N,x)-S(y)-T(N,y)\|\\
\leq~&K(1+4n\Lambda+2n\Lambda L)|x-y|^{\beta_{0}}.
\end{split}
\end{equation}
where we use \textbf{H}1 (1), (3) in the second-to-last line and \textbf{H}2 (1), \textbf{H}3 in the last line.
By direct calculation, it is clear that
\begin{equation}\label{matrix lemma equation2}
\begin{split}
&|\Psi_{H,q}(N,y)-\Psi_{H,q}(\tilde{N},y)|\\
\leq~&|(\Psi_{H,q})_{ij}(\bar{N},y)|\cdot|N_{ij}-\tilde{N}_{ij}|\\
\leq~&|a_{ij}(\bar{N},y)-a_{ij}(H,q)|\cdot|N_{ij}-\tilde{N}_{ij}|,
\end{split}
\end{equation}
where in the second line we use the mean value theorem and $\bar{N}$ is on the segment with endpoints $N$ and $\tilde{N}$. For any $\varepsilon>0$, there exists a constant $\delta$ depending only on $\varepsilon$ and the moduli of continuity of $a_{ij}$ with the following property. Suppose
\begin{equation*}
\|\bar{N}-H\|+|y-q|<\delta,
\end{equation*}
then we have
\begin{equation*}
|a_{ij}(\bar{N},y)-a_{ij}(H,q)|<\varepsilon.
\end{equation*}
By the assumptions of $F$, $S$ and $T$, we can check that the moduli of continuity of $a_{ij}$ depends on $n$, $\Lambda$, $K$, $L$, $\beta_{0}$ and the moduli of continuity of $F_{ij}$. Therefore, there exists a constant $\delta$ depending only on $n$, $\Lambda$, $K$, $L$, $\beta_{0}$ and the moduli of continuity of $F_{ij}$ such that, if
\begin{equation*}
\|N-H\|+\|\tilde{N}-H\|+|x-q|+|y-q|<\delta,
\end{equation*}
then we have
\begin{equation*}
\begin{split}
|\Psi_{H,q}(N,x)-\Psi_{H,q}(\tilde{N},y)|&\leq|\Psi_{H,q}(N,x)-\Psi_{H,q}(N,y)|+|\Psi_{H,q}(N,y)-\Psi_{H,q}(\tilde{N},y)|\\
&\leq\varepsilon\|N-\tilde{N}\|+K(1+4n\Lambda+2n\Lambda L)|x-y|^{\beta_{0}},
\end{split}
\end{equation*}
where we use (\ref{matrix lemma equation1}) and (\ref{matrix lemma equation2}) in the second line.
\end{proof}

\begin{lemma}\label{local estimate}
Under the assumptions of Theorem \ref{PDE Theorem}, there exists a positive number $r$ such that for any $B_{q}(r)\subset B_{0}(\frac{3}{4})$, we have the following estimate
\begin{equation*}
[D^{2}u]_{C^{\gamma_{0}}(B_{q}(\frac{r}{2}))}\leq C,
\end{equation*}
where $r$ and $C$ depend only on $\alpha_{0}$, $K$, $n$, $\Lambda$, $\lambda$, $\beta_{0}$, $\|f\|_{C^{\alpha_{0}}(B_{0}(1))}$, $\|u\|_{L^{\infty}(B_{0}(1))}$ and the moduli of continuity of $F_{ij}$.
\end{lemma}

\begin{proof}
By Theorem \ref{TWWY PDE Theorem}, we have
\begin{equation*}
\|u\|_{C^{2,\alpha}(B_{0}(\frac{3}{4}))}\leq C_{1},
\end{equation*}
where $\alpha$ and $C_{1}$ depend on $\alpha_{0}$, $K$, $n$, $\Lambda$, $\lambda$, $\beta_{0}$, $\|f\|_{C^{\alpha_{0}}(B_{0}(1))}$ and $\|u\|_{L^{\infty}(B_{0}(1))}$. For any $B_{q}(r)\subset B_{0}(\frac{3}{4})$, it is clear that
\begin{equation}\label{local estimate equation1}
\sup_{x,y\in B_{q}(r)}\|D^{2}u(x)-D^{2}u(y)\|\leq 2^{\alpha}r^{\alpha}[D^{2}u]_{C^{\alpha}(B_{0}(\frac{3}{4}))}.
\end{equation}
By (\ref{local estimate equation1}) and Lemma \ref{matrix lemma}, for any $\varepsilon>0$, we can choose $r$ sufficiently small such that for all $x,y\in B_{q}(r)$, we have
\begin{equation*}
\|D^{2}u(x)-D^{2}u(q)\|+\|D^{2}u(y)-D^{2}u(q)\|+|x-q|+|y-q|<\delta,
\end{equation*}
which implies
\begin{equation*}
|\Psi_{D^{2}u(q),q}(D^{2}u(x),x)-\Psi_{D^{2}u(q),q}(D^{2}u(y),y)|\leq\varepsilon\|D^{2}u(x)-D^{2}u(y)\|+L_{0}|x-y|^{\beta_{0}},
\end{equation*}
where $L_{0}=K(1+4n\Lambda+2n\Lambda L)$. It then follows that
\begin{equation}\label{local estimate equation4}
|f(x)-f(y)-a_{ij}(D^{2}u(q),q)(u_{ij}(x)-u_{ij}(y))|\leq\varepsilon\|D^{2}u(x)-D^{2}u(y)\|+L_{0}|x-y|^{\beta_{0}}.
\end{equation}
By Lemma \ref{regularity proposition}, we have $u\in C^{2,\gamma}(B_{0}(\frac{3}{4}))$ for any $\gamma\in(0,\gamma_{0})$. Thus, (\ref{local estimate equation4}) implies that
\begin{equation}\label{local estimate equation2}
[a_{ij}(D^{2}u(q),q)u_{ij}]_{C^{\gamma}(B_{q}(r))}^{(2)}\leq\varepsilon [D^{2}u]_{C^{\gamma}(B_{q}(r))}^{(2)}+[f]_{C^{\gamma}(B_{q}(r))}^{(2)}+2K(1+4n\Lambda+2n\Lambda L).
\end{equation}
By Lemma \ref{uniformly elliptic} and the Schauder estimate, we obtain
\begin{equation}\label{local estimate equation3}
[D^{2}u]_{C^{\gamma}(B_{q}(r))}^{(2)}\leq C_{\gamma}\left([a_{ij}(D^{2}u(q),q)u_{ij}]_{C^{\gamma}(B_{q}(r))}^{(2)}+\|u\|_{L^{\infty}(B_{q}(r))}\right),
\end{equation}
where $C_{\gamma}$ depends only on $n$, $\gamma$, $K^{-1}\lambda$ and $\sqrt{2n}K\Lambda$. Combining (\ref{local estimate equation2}) and (\ref{local estimate equation3}), we have
\begin{equation*}
\begin{split}
[D^{2}u]_{C^{\gamma}(B_{q}(r))}^{(2)}\leq&~ C_{\gamma}\varepsilon [D^{2}u]_{C^{\gamma}(B_{q}(r))}^{(2)}+C_{\gamma}[f]_{C^{\gamma}(B_{q}(r))}^{(2)}\\
&+2KC_{\gamma}(1+4n\Lambda+2n\Lambda L)+C_{\gamma}\|u\|_{L^{\infty}(B_{q}(r))}.
\end{split}
\end{equation*}
Taking $\varepsilon=\frac{1}{2C_{\gamma}}$, we can obtain
\begin{equation}\label{local estimate equation5}
[D^{2}u]_{C^{\gamma}(B_{q}(r))}^{(2)}\leq 2C_{\gamma}[f]_{C^{\gamma}(B_{q}(r))}^{(2)}+4KC_{\gamma}(1+4n\Lambda+2n\Lambda L)+2C_{\gamma}\|u\|_{L^{\infty}(B_{q}(r))}.
\end{equation}
It is clear that $r$ depends on $\varepsilon$ and $\varepsilon$ depends on $\gamma$. Thus, we use $\varepsilon_{\gamma}$ to denote $\varepsilon$ and $r_{\gamma}$ to denote $r$. Since $C_{\gamma}$ is the constant in the Schauder estimate and $\gamma_{0}\in(0,1)$, we have
\begin{equation*}
\lim_{\gamma\rightarrow\gamma_{0}}C_{\gamma}=C_{\gamma_{0}}<\infty,
\end{equation*}
which implies
\begin{equation*}
\lim_{\gamma\rightarrow\gamma_{0}}\varepsilon_{\gamma}=\varepsilon_{\gamma_{0}}>0~~and~~\lim_{\gamma\rightarrow\gamma_{0}}r_{\gamma}=r_{\gamma_{0}}>0.
\end{equation*}
For convenience, we still use $r$ to denote the limit $r_{\gamma_{0}}$ in the following. Then by letting $\gamma\rightarrow\gamma_{0}$ in (\ref{local estimate equation5}), we get
\begin{equation*}
[D^{2}u]_{C^{\gamma_{0}}(B_{q}(r))}^{(2)}\leq 2C_{\gamma_{0}}[f]_{C^{\gamma_{0}}(B_{q}(r))}^{(2)}+4KC_{\gamma_{0}}(1+4n\Lambda+2n\Lambda L)+2C_{\gamma_{0}}\|u\|_{L^{\infty}(B_{q}(r))},
\end{equation*}
which implies $u\in C^{2,\gamma_{0}}(B_{q}(\frac{r}{2}))$ and
\begin{equation*}
[D^{2}u]_{C^{\gamma_{0}}(B_{q}(\frac{r}{2}))}\leq C,
\end{equation*}
where $r$ and $C$ depend only on $\alpha_{0}$, $K$, $n$, $\Lambda$, $\lambda$, $\beta_{0}$, $\|f\|_{C^{\alpha_{0}}(B_{0}(1))}$, $\|u\|_{L^{\infty}(B_{0}(1))}$ and the moduli of continuity of $F_{ij}$.
\end{proof}

\begin{proof}[Proof of Theorem \ref{PDE Theorem}]
We now take any $x,y\in B_{0}(\frac{1}{2})$. Let $q=(x+y)/2$ and $r$ be the constant that we obtain from Lemma \ref{local estimate}. If $|x-y|<r$, we have
\begin{equation*}
\frac{|D^{2}u(x)-D^{2}u(y)|}{|x-y|^{\gamma_{0}}}\leq [D^{2}u]_{C^{\gamma_{0}}(B_{q}(\frac{r}{2}))}.
\end{equation*}
If $|x-y|\geq r$, then
\begin{equation*}
\frac{|D^{2}u(x)-D^{2}u(y)|}{|x-y|^{\gamma_{0}}}\leq2r^{-\gamma_{0}}\|D^{2}u\|_{L^{\infty}(B_{q}(\frac{r}{2}))}.
\end{equation*}
Hence, by Lemma \ref{local estimate} and interpolation inequalities, we complete the proof.
\end{proof}

\section{The proof of Theorem \ref{New Conical Theorem}}
For the reader's convenience, let us recall some definitions of cone metric. We consider the model cone metric on $B_{0}(1)\subset\mathbb{C}^{n}$
\begin{equation*}
\omega_{\beta}=\sqrt{-1}\frac{\beta^{2}}{|z|^{2-2\beta}}dz\wedge d\bar{z}+\sqrt{-1}\sum_{k=2}^{n}dz_{k}\wedge d\bar{z}_{k},
\end{equation*}
where $(z,z_{2},\cdot\cdot\cdot,z_{n})$ is the standard coordinates of $B_{0}(1)\subset\mathbb{C}^{n}$.

Let $\xi=|z|^{\beta-1}z$. For any function $f(z,z_{2},\cdot\cdot\cdot,z_{n})$ on $B_{0}(1)$, we let
\begin{equation*}
\tilde{f}(\xi,z_{2},\cdot\cdot\cdot,z_{n})=f(z,z_{2},\cdot\cdot\cdot,z_{n}).
\end{equation*}
We write
\begin{equation*}
C^{\alpha,\beta}=\{f\in L^{\infty}|\tilde{f}\in C^{\alpha}\}
\end{equation*}
and
\begin{equation*}
C_{0}^{\alpha,\beta}=\{f\in C^{\alpha,\beta}|f(0,z_{2},\cdot\cdot\cdot,z_{n})=0\}.
\end{equation*}
A (1,0)-form $\tau$ is said to be of class $C^{\alpha,\beta}$ if
\begin{equation*}
\tau\left(\frac{\partial}{\partial z_{k}}\right)\in C^{\alpha,\beta}~~and~~|z|^{1-\beta}\tau\left(\frac{\partial}{\partial z}\right)\in C_{0}^{\alpha,\beta}~~k=2,\cdot\cdot\cdot,n.
\end{equation*}
A (1,1)-form $\sigma$ is said to be of class $C^{\alpha,\beta}$ if
\begin{equation*}
\sigma\left(\frac{\partial}{\partial z_{k}},\frac{\partial}{\partial \bar{z}_{l}}\right)\in C^{\alpha,\beta},~~|z|^{1-\beta}\sigma\left(\frac{\partial}{\partial z_{k}},\frac{\partial}{\partial \bar{z}}\right)\in C_{0}^{\alpha,\beta},~~|z|^{1-\beta}\sigma\left(\frac{\partial}{\partial z},\frac{\partial}{\partial \bar{z}_{l}}\right)\in C_{0}^{\alpha,\beta}
\end{equation*}
and
\begin{equation*}
|z|^{2-2\beta}\sigma\left(\frac{\partial}{\partial z},\frac{\partial}{\partial \bar{z}}\right)\in C^{\alpha,\beta}~~k,l=2,\cdot\cdot\cdot,n.
\end{equation*}
We write
\begin{equation*}
C^{2,\alpha,\beta}=\{f\in C^{2}(B_{0}(1)\setminus D)|f,\partial f,\partial\bar{\partial}f~~are~of~class~C^{2,\alpha,\beta}\},
\end{equation*}
where $D=\{(z,z_{2},\cdot\cdot\cdot,z_{n})\in \mathbb{C}^{n}|z=0\}$. For more details of cone metric, see Section 2 of \cite{Br}. Now, let us prove Theorem \ref{New Conical Theorem}. First, We have the following proposition.
\begin{proposition}\label{Conical Proposition}
Let $\omega_{C}=C\omega_{\beta}\bar{C}^{T}$ ($C$ is a constant matrix in $\mathbb{C}^{n\times n}$) and $\lambda\omega_{\beta}\leq\omega_{C}\leq\Lambda\omega_{\beta}$. Let $\alpha_{0}\in(0,\min(\frac{1}{\beta}-1,1))$ be a constant. There exist a positive number $\delta_{C}$ and $C_{C}$, both depending only on $\lambda$, $\Lambda$, $n$, $\beta$ and $\alpha_{0}$, with the following property. Suppose $\phi\in C^{2,\gamma,\beta}(B_{0}(1))$ for any $\gamma\in(0,\alpha_{0})$, $\sqrt{-1}\partial\bar{\partial}\phi$ is a conical K\"{a}hler metric and $f\in C^{\alpha_{0},\beta}(B_{0}(1))$. Suppose
\begin{equation*}
\det\phi_{i\bar{j}}=e^{f}~~and~~\frac{\omega_{C}}{1+\delta_{C}}\leq\sqrt{-1}\partial\bar{\partial}\phi\leq(1+\delta_{C})\omega_{C}~~over~B_{0}(1)\backslash\{0\}\times\mathbb{C}^{n-1},
\end{equation*}
then we have $\phi\in C^{2,\alpha_{0},\beta}(B_{0}(\frac{1}{2}))$ and the following estimate
\begin{equation*}
[\sqrt{-1}\partial\bar{\partial}\phi]_{C^{\alpha_{0},\beta}(B_{0}(\frac{1}{2}))}\leq C_{C}\left([e^{f}]_{C^{\alpha_{0},\beta}(B_{0}(1))}+\|\phi\|_{L^{\infty}(B_{0}(1))}\right).
\end{equation*}
\end{proposition}
Given this proposition, we can prove Theorem \ref{New Conical Theorem} by the same argument in the proof of Theorem \ref{Main Theorem}. Therefore, it suffices to prove Proposition \ref{Conical Proposition}. In fact, the proof of Proposition \ref{Conical Proposition} is very similar to the proof of Proposition 5.2 in \cite{CW1402}. For the reader's convenience, we give a proof here.

First, for any positive definite matrix $G$ such that
\begin{equation*}
\lambda I_{n}\leq G\leq\Lambda I_{n},
\end{equation*}
we define a function of matrix
\begin{equation*}
F_{G}(M)=\det M-\det G\cdot tr(G^{-1}M).
\end{equation*}

Then we have the following lemma.
\begin{lemma}\label{Conical matrix lemma}
For any $\varepsilon>0$, there exists a positive number $\delta$ depending only on $\varepsilon$, $n$, $\lambda$ and $\Lambda$, with the following property. If we have
\begin{equation*}
\|M-G\|+\|\tilde{M}-G\|<\delta,
\end{equation*}
then
\begin{equation*}
|F_{G}(M)-F_{G}(\tilde{M})|\leq\varepsilon\|M-\tilde{M}\|.
\end{equation*}
\end{lemma}

\begin{proof}
First, we observe that $dF_{I}(I)=0$. Thus, for any $\varepsilon>0$, there exists a positive number $\delta_{I}=\delta_{I}(\varepsilon,n)$ with the following property. If we have
\begin{equation*}
\|N-I\|+\|\tilde{N}-I\|<\delta_{I},
\end{equation*}
then
\begin{equation}\label{Conical matrix lemma equation1}
|F_{I}(N)-F_{I}(\tilde{N})|\leq\varepsilon\|N-\tilde{N}\|.
\end{equation}
Let $\delta=\frac{\lambda\delta_{I}}{\sqrt{n}}$. Hence, if
\begin{equation*}
\|M-G\|+\|\tilde{M}-G\|<\delta,
\end{equation*}
then
\begin{equation*}
\|G^{-1}M-I\|+\|G^{-1}\tilde{M}-I\|\leq\|G^{-1}\|\delta\leq\delta_{I}.
\end{equation*}
By (\ref{Conical matrix lemma equation1}), we obtain
\begin{equation*}
\begin{split}
|F_{G}(M)-F_{G}(\tilde{M})|&=(\det G)|F_{I}(G^{-1}M)-F_{I}(G^{-1}M)|\\
&\leq\varepsilon\Lambda^{n}\|G^{-1}(M-\tilde{M})\|\\
&\leq\frac{\sqrt{n}\varepsilon\Lambda^{n}}{\lambda}\|M-\tilde{M}\|.
\end{split}
\end{equation*}
We complete the proof.
\end{proof}

\begin{proof}[Proof of Proposition \ref{Conical Proposition}]
We compute under the singular coordinates. By Lemma \ref{Conical matrix lemma}, for any $\varepsilon>0$, there exists a constant $\delta_{C}$ depending on $\varepsilon$ with the following property. When
\begin{equation*}
(1-\delta_{C})C\bar{C}^{T}\leq\sqrt{-1}\partial\bar{\partial}\phi\leq(1+\delta_{C})C\bar{C}^{T},
\end{equation*}
we have
\begin{equation*}
|F_{C\bar{C}^{T}}(\sqrt{-1}\partial\bar{\partial}\phi(x))-F_{C\bar{C}^{T}}(\sqrt{-1}\partial\bar{\partial}\phi(y))|\leq\varepsilon\|\sqrt{-1}\partial\bar{\partial}\phi(x)-\sqrt{-1}\partial\bar{\partial}\phi(y)\|,
\end{equation*}
which implies
\begin{equation*}
[\det(\sqrt{-1}\partial\bar{\partial}\phi)-(\det(C\bar{C}^{T}))\Delta_{C\bar{C}^{T}}\phi]_{C^{\gamma,\beta}(B_{0}(1))}^{(2)}\leq\varepsilon[\sqrt{-1}\partial\bar{\partial}\phi]_{C^{\gamma,\beta}(B_{0}(1))}^{(2)}.
\end{equation*}
Since $\det(\sqrt{-1}\partial\bar{\partial}\phi)=e^{f}$, we deduce that
\begin{equation}\label{Conical Proposition equation1}
[(\det(C\bar{C}^{T}))\Delta_{C\bar{C}^{T}}\phi]_{C^{\gamma,\beta}(B_{0}(1))}^{(2)}\leq\varepsilon [\sqrt{-1}\partial\bar{\partial}\phi]_{C^{\gamma,\beta}(B_{0}(1))}^{(2)}+[e^{f}]_{C^{\gamma,\beta}(B_{0}(1))}^{(2)}.
\end{equation}
By the conic Schauder estimate, we obtain
\begin{equation}\label{Conical Proposition equation2}
[\sqrt{-1}\partial\bar{\partial}\phi]_{C^{\gamma,\beta}(B_{0}(1))}^{(2)}\leq C_{\gamma}\left([(\det(C\bar{C}^{T}))\Delta_{C\bar{C}^{T}}\phi]_{C^{\gamma,\beta}(B_{0}(1))}^{(2)}+\|\phi\|_{L^{\infty}(B_{0}(1))}\right).
\end{equation}
Combining (\ref{Conical Proposition equation1}) and (\ref{Conical Proposition equation2}), we can get
\begin{equation*}
[\sqrt{-1}\partial\bar{\partial}\phi]_{C^{\gamma,\beta}(B_{0}(1))}^{(2)}\leq C_{\gamma}\varepsilon[\sqrt{-1}\partial\bar{\partial}\phi]_{C^{\gamma,\beta}(B_{0}(1))}^{(2)}+C_{\gamma}[e^{f}]_{C^{\gamma,\beta}(B_{0}(1))}^{(2)}+C_{\gamma}\|\phi\|_{L^{\infty}(B_{0}(1))},
\end{equation*}
where $C_{\gamma}$ depends only on $\lambda$, $\Lambda$, $n$, $\gamma$ and $\beta$. Taking $\varepsilon=\frac{1}{2C_{\gamma}}$, it is clear that
\begin{equation}\label{Conical Proposition equation3}
[\sqrt{-1}\partial\bar{\partial}\phi]_{C^{\gamma,\beta}(B_{0}(1))}^{(2)}\leq 2C_{\gamma}[e^{f}]_{C^{\gamma,\beta}(B_{0}(1))}^{(2)}+2C_{\gamma}\|\phi\|_{L^{\infty}(B_{0}(1))}.
\end{equation}
We know that $\delta_{C}$ depends on $\varepsilon$ and $\varepsilon$ depends on $\gamma$. Thus, we use $\varepsilon_{\gamma}$ to denote $\varepsilon$ and $\delta_{\gamma}$ to denote $\delta_{C}$. Since $C_{\gamma}$ is the constant in the conic Schauder estimate and $\alpha_{0}\in(0,\min(\frac{1}{\beta}-1,1))$, we have
\begin{equation*}
\lim_{\gamma\rightarrow\alpha_{0}}C_{\gamma}=C_{\alpha_{0}}<\infty,
\end{equation*}
which implies
\begin{equation*}
\lim_{\gamma\rightarrow\alpha_{0}}\varepsilon_{\gamma}=\varepsilon_{\alpha_{0}}>0~~and~~\lim_{\gamma\rightarrow\alpha_{0}}\delta_{\gamma}=\delta_{\alpha_{0}}>0.
\end{equation*}
For convenience, we still use $\delta_{C}$ to denote the limit $\delta_{\alpha_{0}}$. Then by letting $\gamma\rightarrow\alpha_{0}$ in (\ref{Conical Proposition equation3}), we get
\begin{equation*}
[\sqrt{-1}\partial\bar{\partial}\phi]_{C^{\alpha_{0},\beta}(B_{0}(1))}^{(2)}\leq 2C_{\alpha_{0}}[e^{f}]_{C^{\alpha_{0},\beta}(B_{0}(1))}^{(2)}+2C_{\alpha_{0}}\|\phi\|_{L^{\infty}(B_{0}(1))},
\end{equation*}
which implies $\phi\in C^{2,\alpha_{0},\beta}(B_{0}(\frac{1}{2}))$ and
\begin{equation*}
[\sqrt{-1}\partial\bar{\partial}\phi]_{C^{\alpha_{0},\beta}(B_{0}(\frac{1}{2}))}\leq C_{C}\left([e^{f}]_{C^{\alpha_{0},\beta}(B_{0}(1))}+\|\phi\|_{L^{\infty}(B_{0}(1))}\right).
\end{equation*}
\end{proof}

\section{More general elliptic equations}
In this section, we consider more general elliptic equations. First of all, let us recall the definition 3.2 in \cite{TWWY} first.
\begin{definition}
Let $\mathcal{F}_{n}(\lambda,\Lambda,K,\beta_{0})$ be a family of functions $\Phi:Sym(n)\times B_{0}(1)\rightarrow \mathbb{R}$ depending on positive constants $\lambda$, $\Lambda$, $K$ and $\beta_{0}\in(0,1)$. An element $\Phi\in\mathcal{F}_{n}(\lambda,\Lambda,K,\beta_{0})$ satisfies the following conditions:
\begin{enumerate}[$\bullet$]
    \item Fiberwise concavity. For each fixed $x\in B_{0}(1)$,
    \begin{equation*}
      \Phi\left(\frac{A+B}{2},x\right)\geq\frac{1}{2}\Phi(A,x)+\frac{1}{2}\Phi(B,x),
    \end{equation*}
    for all $A,B\in Sym(n)$.

    \item Uniform Ellipticity. For all $x\in B_{0}(1)$ and all $N,P\in Sym(n)$ with $P\geq0$ we have
    \begin{equation*}
    \lambda\|P\|\leq\Phi(N+P,x)-\Phi(N,x)\leq\Lambda\|P\|.
    \end{equation*}

    \item H\"{o}lder bound in $x$. For all $x,y\in B_{0}(1)$ and all $N\in Sym(n)$,
    \begin{equation*}
    \frac{|\Phi(N,x)-\Phi(N,y)|}{\|N\|+1}\leq K|x-y|^{\beta_{0}}.
    \end{equation*}
\end{enumerate}
\end{definition}
Next, we recall the Evans-Krylov theorem for $\Phi\in\mathcal{F}_{n}(\lambda,\Lambda,K,\beta_{0})$ (see \cite{Caff,TWWY}).
\begin{theorem}\label{TWWY Evans-Krylov Theorem}
Assume that $\Phi\in\mathcal{F}_{n}(\lambda,\Lambda,K,\beta_{0})$ and $f\in C^{\alpha_{0}}(B_{0}(1))$. If $u\in C^{0}(B_{0}(1))$ is a viscosity solution of the equation
\begin{equation*}
\Phi(D^{2}u(x),x)=f(x)~~in~B_{0}(1).
\end{equation*}
Then $u\in C^{2,\alpha}(B_{0}(\frac{1}{2}))$ and
\begin{equation*}
\|u\|_{C^{2,\alpha}(B_{0}(\frac{1}{2}))}\leq C,
\end{equation*}
where $\alpha$ and C depend only on $K$, $\alpha_{0}$, $\beta_{0}$, $n$, $\lambda$, $\Lambda$, $\|f\|_{C^{\alpha_{0}}(B_{0}(1))}$, $\|u\|_{L^{\infty}(B_{0}(1))}$ and $\Phi(0,0)$.
\end{theorem}

The following Theorem is our result, which improves Theorem \ref{TWWY Evans-Krylov Theorem}.
\begin{theorem}\label{Evans-Krylov Theorem}
Assume that $\Phi\in\mathcal{F}_{n}(\lambda,\Lambda,K,\beta_{0})$ is of class $C^{1}$ in its domain and $f\in C^{\alpha_{0}}(B_{0}(1))$. If $u\in C^{0}(B_{0}(1))$ is a viscosity solution of the equation
\begin{equation*}
\Phi(D^{2}u(x),x)=f(x)~~in~B_{0}(1).
\end{equation*}
Then $u\in C^{2,\gamma_{0}}(B_{0}(\frac{1}{2}))$ and
\begin{equation*}
\|u\|_{C^{2,\gamma_{0}}(B_{0}(\frac{1}{2}))}\leq C,
\end{equation*}
where $\gamma_{0}=\min(\alpha_{0},\beta_{0})$ and C depends only on $K$, $\alpha_{0}$, $\beta_{0}$, $n$, $\lambda$, $\Lambda$, $\|f\|_{C^{\alpha_{0}}(B_{0}(1))}$, $\|u\|_{L^{\infty}(B_{0}(1))}$, $\Phi(0,0)$ and the moduli of continuity of $\Phi_{ij}$.
\end{theorem}

Compared with Theorem \ref{TWWY Evans-Krylov Theorem}, our result is optimal regarding the H\"{o}lder exponent. And we can prove $u\in C^{2,\gamma_{0}}(B_{0}(\frac{1}{2}))$ without using the condition that $\Phi$ is concave. However, we need $\Phi$ to be $C^{1}$ in Theorem \ref{Evans-Krylov Theorem} and the constant $C$ also depends on the moduli of continuity of $\Phi_{ij}$.

We give a sketch of Theorem \ref{Evans-Krylov Theorem}. By Theorem \ref{TWWY Evans-Krylov Theorem}, we have the estimate of $D^{2}u$, which implies
\begin{equation*}
-LI_{2n}\leq D^{2}u(x)\leq LI_{2n}~~\forall x\in B_{0}(\frac{3}{4}),
\end{equation*}
where $L$ depends only on $\alpha_{0}$, $K$, $\beta_{0}$, $n$, $\lambda$, $\Lambda$, $\|f\|_{C^{\alpha_{0}}(B_{0}(1))}$, $\|u\|_{L^{\infty}(B_{0}(1))}$ and $\Phi(0,0)$. For any $H\in P[-L,L]$ and $q\in B_{0}(1)$, we define
\begin{equation*}
G_{H,q}(M,x)=\Phi(M,x)-\Phi_{ij}(H,q)m_{ij}.
\end{equation*}
Since the proof of Theorem \ref{Evans-Krylov Theorem} is very similar to the proof of Theorem \ref{PDE Theorem}, it suffices to prove the following lemmas. In fact, Lemma \ref{Evans-Krylov Lemma1} and Lemma \ref{Evans-Krylov Lemma2} are analogues of Lemma \ref{matrix lemma} and Lemma \ref{local estimate}, we can prove them by the same argument.
\begin{lemma}\label{Evans-Krylov Lemma1}
For any $\varepsilon>0$, there exists a positive number $\delta$ depending only on $\varepsilon$ and the moduli of continuity of $\Phi_{ij}$, with the following property. For any $H, N, \tilde{N}\in P[-L,L]$ and $q, x, y\in \overline{B_{0}(\frac{3}{4})}$, if
\begin{equation*}
\|N-H\|+\|\tilde{N}-H\|+|x-q|+|y-q|<\delta,
\end{equation*}
then we have
\begin{equation*}
|G_{H,q}(N,x)-G_{H,q}(\tilde{N},y)|\leq\varepsilon\|N-\tilde{N}\|+K(L+1)|x-y|^{\beta_{0}}.
\end{equation*}
\end{lemma}

\begin{lemma}\label{Evans-Krylov Lemma2}
Under the assumptions of Theorem \ref{Evans-Krylov Theorem}, there exists a positive number $r$ such that for any $B_{q}(r)\subset B_{0}(\frac{3}{4})$, we have the following estimate
\begin{equation*}
[D^{2}u]_{C^{\gamma_{0}}(B_{q}(\frac{r}{2}))}\leq C,
\end{equation*}
where $r$ and $C$ depend only on $\alpha_{0}$, $K$, $\beta_{0}$, $n$, $\lambda$, $\Lambda$, $\|f\|_{C^{\alpha_{0}}(B_{0}(1))}$, $\|u\|_{L^{\infty}(B_{0}(1))}$, $\Phi(0,0)$ and the moduli of continuity of $\Phi_{ij}$.
\end{lemma}

\section{Parabolic equations}
In this section, we consider parabolic equations. Let us recall the standard conventions first. Let $D$ and $D_{t}$ represent the partial differentiation with respect to $x$-variables and $t$-variables respectively. Let $Q_{(x,t)}(r)$ represent the domain $B_{x}(r)\times(t-r^{2},t]$. A function $u$ defined on $Q_{(0,0)}(1)$ is said to be $C^{k,\alpha}$ with $k$ being even if
\begin{equation*}
\|u\|_{C^{k,\alpha},Q_{(0,0)}(1)}:=\sum_{0\leq i+2j\leq k}\|D^{i}D_{t}^{j}u\|_{L^{\infty},Q_{(0,0)}(1)}+\|D^{k}u\|_{C^{\alpha},Q_{(0,0)}(1)}+\|D_{t}^{\frac{k}{2}}u\|_{C^{\frac{\alpha}{2}},Q_{(0,0)}(1)}<\infty.
\end{equation*}

By the same method (in Section 2 of \cite{TWWY}), we can use the parabolic version of Theorem \ref{PDE Theorem} to establish the parabolic version of Theorem \ref{Main Theorem}. Thus, we need to state the parabolic version of Theorem \ref{PDE Theorem}.

We consider parabolic equation of the form
\begin{equation*}
u_{t}(x,t)-F\left(S(x,t)+T(D^{2}u(x,t),x,t),x,t\right)=f(x,t)~~in~Q_{(0,0)}(1),
\end{equation*}
where $f\in C^{\alpha_{0}}(Q_{(0,0)}(1))$ and
\begin{equation*}
\begin{split}
&F:Sym(2n)\times Q_{(0,0)}(1)\rightarrow \mathbb{R};\\
&S:Q_{(0,0)}(1)\rightarrow Sym(2n);\\
&T:Sym(2n)\times Q_{(0,0)}(1)\rightarrow Sym(2n).
\end{split}
\end{equation*}
We assume that there exists a compact convex set $\mathcal{E}\subset Sym(2n)$, positive constants $\lambda$, $\Lambda$, $K$ and $\beta_{0}\in(0,1)$ such that the following hold.
\begin{enumerate}[\textbf{H}1:]
    \item  $F$ is of class $C^{1}$ in $U\times Q_{(0,0)}(1)$ where $U$ is a neighborhood of $\mathcal{E}$ and
           \begin{enumerate}[(1)]
           \item $F$ is uniformly elliptic in $\mathcal{E}$:
                 \begin{equation*}
                 \lambda|\xi|^{2}\leq F_{ij}(M,x,t)\xi^{i}\xi^{j}\leq\Lambda|\xi|^{2},
                 \end{equation*}
                 for all $M\in\mathcal{E}$, $(x,t)\in Q_{(0,0)}(1)$ and $\xi\in\mathbb{R}^{2n}$, where $F_{ij}(M,x,t)=\frac{\partial F}{\partial m_{ij}}(M,x,t)$.
           \item $F$ is concave in $\mathcal{E}$:
                 \begin{equation*}
                 F(\frac{A+B}{2},x,t)\geq\frac{1}{2}F(A,x,t)+\frac{1}{2}F(B,x,t),
                 \end{equation*}
                 for all $A,B\in\mathcal{E}$ and $(x,t)\in Q_{(0,0)}(1)$.
           \item $F$ has the following uniform H\"{o}lder bound in $(x,t)$:
                 \begin{equation*}
                 |F(N,x,t_{x})-F(N,y,t_{y})|\leq K|x-y|^{\beta_{0}}+K|t_{x}-t_{y}|^{\frac{\beta_{0}}{2}}~~and~~|F(N,0,0)|\leq K,
                 \end{equation*}
                 for all $N\in\mathcal{E}$ and $(x,t_{x}),(y,t_{y})\in Q_{(0,0)}(1)$.
           \end{enumerate}

    \item  The map $T:Sym(2n)\times Q_{(0,0)}(1)\rightarrow Sym(2n)$ satisfies the following conditions:
           \begin{enumerate}[(1)]
           \item For all $(x,t_{x}),(y,t_{y})\in Q_{(0,0)}(1)$ and all $N\in Sym(2n)$,
                 \begin{equation*}
                 \frac{\|T(N,x,t_{x})-T(N,y,t_{y})\|}{\|N\|+1}\leq K|x-y|^{\beta_{0}}+K|t_{x}-t_{y}|^{\frac{\beta_{0}}{2}}.
                 \end{equation*}
           \item For each fixed $(x,t)\in Q_{(0,0)}(1)$, the map $M\mapsto T(M,x,t)$ is linear on $Sym(2n)$. For convenience, we
                 assume $T_{ij}(M,x,t)=T_{ij,kl}(x,t)m_{kl}$.
           \item For all $P\geq0$ and $(x,t)\in Q_{(0,0)}(1)$,
                 \begin{equation*}
                 T(P,x,t)\geq0~~and~~K^{-1}\|P\|\leq\|T(P,x,t)\|\leq K\|P\|.
                 \end{equation*}
           \end{enumerate}

    \item  $S:Q_{(0,0)}(1)\rightarrow Sym(2n)$ has a uniform $C^{\beta_{0}}$ bound:
           \begin{equation*}
           \|S(x,t_{x})-S(y,t_{y})\|\leq K|x-y|^{\beta_{0}}+K|t_{x}-t_{y}|^{\frac{\beta_{0}}{2}}~~and~~\|S(0,0)\|\leq K,
           \end{equation*}
           for all $(x,t_{x}),(y,t_{y})\in Q_{(0,0)}(1)$.
\end{enumerate}

The following Theorem is the parabolic version of Theorem \ref{PDE Theorem}.
\begin{theorem}\label{Parabolic PDE Theorem}
With the assumption above, suppose that $u\in C^{2}(Q_{(0,0)}(1))$ solves
\begin{equation*}
u_{t}(x,t)-F\left(S(x,t)+T(D^{2}u(x,t),x,t),x,t\right)=f(x,t)~~in~Q_{(0,0)}(1),
\end{equation*}
and satisfies
\begin{equation*}
S(x,t)+T(D^{2}u(x,t),x,t)\in\mathcal{E},~~\forall~(x,t)\in Q_{(0,0)}(1).
\end{equation*}
Then $u\in C^{2,\gamma_{0}}(Q_{(0,0)}(\frac{1}{2}))$ and
\begin{equation*}
\|u\|_{C^{2,\gamma_{0}}(Q_{(0,0)}(\frac{1}{2}))}\leq C,
\end{equation*}
where $\gamma_{0}=\min(\alpha_{0},\beta_{0})$ and C depends only on $\alpha_{0}$, $K$, $n$, $\Lambda$, $\lambda$, $\beta_{0}$, $\|f\|_{C^{\alpha_{0}}(Q_{(0,0)}(1))}$, $\|u\|_{L^{\infty}(Q_{(0,0)}(1))}$ and the moduli of continuity of $F_{ij}$.
\end{theorem}

Next, let us consider more general uniformly parabolic equations. The following definition is the parabolic version of definition 3.2 in \cite{TWWY}.
\begin{definition}
Let $\mathcal{F}_{n}(\lambda,\Lambda,K,\beta_{0})$ be a family of functions $\Phi:Sym(n)\times Q_{(0,0)}(1)\rightarrow \mathbb{R}$ depending on positive constants $\lambda$, $\Lambda$, $K$ and $\beta_{0}\in(0,1)$. An element $\Phi\in\mathcal{F}_{n}(\lambda,\Lambda,K,\beta_{0})$ satisfies the following conditions:
\begin{enumerate}[$\bullet$]
    \item Fiberwise concavity. For each fixed $x\in Q_{(0,0)}(1)$,
    \begin{equation*}
      \Phi\left(\frac{A+B}{2},x,t\right)\geq\frac{1}{2}\Phi(A,x,t)+\frac{1}{2}\Phi(B,x,t),
    \end{equation*}
    for all $A,B\in Sym(n)$.

    \item Uniform Ellipticity. For all $x\in Q_{(0,0)}(1)$ and all $N,P\in Sym(n)$ with $P\geq0$ we have
    \begin{equation*}
    \lambda\|P\|\leq\Phi(N+P,x,t)-\Phi(N,x,t)\leq\Lambda\|P\|.
    \end{equation*}

    \item H\"{o}lder bound in $(x,t)$. For all $(x,t_{x}),(y,t_{y})\in Q_{(0,0)}(1)$ and all $N\in Sym(n)$,
    \begin{equation*}
    \frac{|\Phi(N,x,t_{x})-\Phi(N,y,t_{y})|}{\|N\|+1}\leq K|x-y|^{\beta_{0}}+K|t_{x}-t_{y}|^{\frac{\beta_{0}}{2}}.
    \end{equation*}
\end{enumerate}
\end{definition}

The following Theorem is the parabolic version of Theorem \ref{Evans-Krylov Theorem}.
\begin{theorem}\label{Parabolic Evans-Krylov Theorem}
Assume that $\Phi\in\mathcal{F}_{n}(\lambda,\Lambda,K,\beta_{0})$ is of class $C^{1}$ in its domain and $f\in C^{\alpha_{0}}(Q_{(0,0)}(1))$. If $u\in C^{0}(Q_{(0,0)}(1))$ is a viscosity solution of the equation
\begin{equation*}
u_{t}(x,t)-\Phi(D^{2}u(x,t),x,t)=f(x,t)~~in~Q_{(0,0)}(1).
\end{equation*}
Then $u\in C^{2,\gamma_{0}}(Q_{(0,0)}(\frac{1}{2}))$ and
\begin{equation*}
\|u\|_{C^{2,\gamma_{0}}(Q_{(0,0)}(\frac{1}{2}))}\leq C,
\end{equation*}
where $\gamma_{0}=\min(\alpha_{0},\beta_{0})$ and C depends only on $K$, $\alpha_{0}$, $\beta_{0}$, $n$, $\lambda$, $\Lambda$, $\|f\|_{C^{\alpha_{0}}(Q_{(0,0)}(1))}$, $\|u\|_{L^{\infty}(Q_{(0,0)}(1))}$, $\Phi(0,0)$ and the moduli of continuity of $\Phi_{ij}$.
\end{theorem}

The regularities and estimates of solutions in Theorem \ref{Parabolic PDE Theorem} and Theorem \ref{Parabolic Evans-Krylov Theorem} are optimal regarding the H\"{o}lder exponent. Since the proofs of Theorem \ref{Parabolic PDE Theorem} and Theorem \ref{Parabolic Evans-Krylov Theorem} are similar to their elliptic versions, we omit them here. All the parabolic results we need in the proofs of Theorem \ref{Parabolic PDE Theorem} and Theorem \ref{Parabolic Evans-Krylov Theorem} (such as Schauder estimate, $L^{p}$ estimate and so on) can be found in \cite{Lie,Wang}.

\section{An application of Theorem \ref{Main Theorem}}
In this section, we give an application of Theorem \ref{Main Theorem}. Let $(M,J,\Omega)$ be a compact manifold of real dimension $2n$, where $J$ is an almost complex structure and $\Omega$ is a symplectic form taming $J$. We can define a Riemannian metric on $M$ by
\begin{equation*}
g_{\Omega}(X,Y)=\frac{1}{2}\left(\Omega(X,JY)+\Omega(Y,JX)\right).
\end{equation*}
Our result is as follows.
\begin{theorem}\label{Calabi-Yau equation theorem}
Let $\tilde{\omega}\in[\Omega]\in H^{2}(M,\mathbb{R})$ be a symplectic form on $M$ and let $\alpha_{0}\in(0,1)$ be a constant. Suppose $\tilde{\omega}$ is compatible with $J$ and solves the Calabi-Yau equation
\begin{equation}\label{Calabi-Yau equation}
\tilde{\omega}^{n}=\sigma,
\end{equation}
where $\sigma$ is a smooth positive volume form on $M$. Suppose we have
\begin{equation*}
tr_{g_{\Omega}}g_{\tilde{\omega}}\leq C_{0}~~in~~B_{R},
\end{equation*}
where $B_{R}$ is a geodesic $g_{\Omega}$-ball of radius $R$. Then we have the following estimate
\begin{equation*}
\|g_{\tilde{\omega}}\|_{C^{\alpha_{0}}(B_{\frac{R}{2}})}\leq C,
\end{equation*}
where $C$ depends only on $(M,J,\Omega)$, $R$, $\alpha_{0}$, $C_{0}$ and $\|\sigma\|_{C^{\alpha_{0}}(B_{R})}$.
\end{theorem}
As we can see, locally, (\ref{Calabi-Yau equation}) can be transformed into the almost complex Monge-Amp\`{e}re equation (see Section 5 of \cite{TWWY}). Then we can use Theorem \ref{Main Theorem} to obtain Theorem \ref{Calabi-Yau equation theorem}.

\section{Appendix}
In this section, we recall some basic results which are crucial to our proof of Theorem \ref{Main Theorem}. Let $\Omega$ be a bounded connected open set in $\mathbb{R}^{n}$ and let $v\in L^{1}(\Omega)$. For any ball $B_{x_{0}}(r)\subset\Omega$, we define
\begin{equation*}
v_{x_{0},r}=\frac{1}{|B_{x_{0}}(r)|}\int_{B_{x_{0}}(r)}v(x)dx.
\end{equation*}
For any unit vector $e$, $h>0$ and $\alpha\in(0,1)$, we define the $\alpha$ difference quotient in the direction $e$ by
\begin{equation*}
D_{e}^{h,\alpha}v(x)=\frac{v(x+he)-v(x)}{h^{\alpha}}.
\end{equation*}

We have the following theorem.
\begin{theorem}\label{Lp regular theorem}
Let $v\in L^{p}(\Omega)$ for any $p\in(2,\infty)$ and let $\Omega'\subset\subset\Omega$. Suppose $D_{e}^{h,\alpha}v\in L^{p}(\Omega')$ for some $\alpha\in(0,1)$ satisfies
\begin{equation*}
\|D_{e}^{h,\alpha}v\|_{L^{p}(\Omega')}\leq M,
\end{equation*}
for any $0<h<dist(\Omega',\partial\Omega)$ and unit vector $e$, where $M$ is a constant independent of $h$ and $e$. Then we have $v\in C^{\gamma}(\Omega'')$ for any $\gamma\in(0,\alpha)$ and any $\Omega''\subset\subset\Omega'$. And we have the following estimate
\begin{equation*}
\|v\|_{C^{\gamma}(\Omega'')}\leq C,
\end{equation*}
where $C$ depends only on $n$, $\gamma$, $M$, $dist(\Omega',\partial\Omega)$, $diam(\Omega)$, $dist(\Omega'',\partial\Omega')$ and $\|v\|_{L^{1}(\Omega')}$.
\end{theorem}
In order to prove Theorem \ref{Lp regular theorem}, it suffices to prove the following lemmas.

\begin{lemma}\label{mollifier lemma}
Under the assumptions of Theorem \ref{Lp regular theorem}, for any $B_{x_{0}}(r)\subset\Omega'$, we have the following estimate
\begin{equation*}
\int_{B_{x_{0}}(r)}|v(x)-v_{x_{0},r}|^{2}dx\leq C\left(n,\alpha,p,d_{0},diam(\Omega)\right)M^{2}r^{n+2(\alpha-\frac{n}{p})},
\end{equation*}
where $d_{0}=dist(\Omega',\partial\Omega)$.
\end{lemma}

\begin{proof}
For any $\varepsilon\in(0,d_{0})$, we denote by $v_{\varepsilon}$ the regularization of $v$, that is,
\begin{equation*}
v_{\varepsilon}(x)=\varepsilon^{-n}\int_{\mathbb{R}^{n}}v(y)\varphi(\frac{x-y}{\varepsilon})dy=\int_{B_{0}(1)}v(x+\varepsilon y)\varphi(y)dy,
\end{equation*}
for any $x\in\Omega'$, where $\varphi$ is a mollifier. For any $B_{x_{0}}(r)\subset\Omega'$, we compute
\begin{equation*}
\begin{split}
\|v_{\varepsilon}-v\|_{L^{p}(B_{x_{0}}(r))}^{p}&=\int_{B_{x_{0}}(r)}\left|\int_{B_{0}(1)}\left(v(x+\varepsilon y)-v(x)\right)\varphi(y)dy\right|^{p}dx\\
&\leq C(n,p)\int_{B_{0}(1)}\left(\int_{B_{x_{0}}(r)}\left|v(x+\varepsilon y)-v(x)\right|^{p}\varphi(y)^{p}dx\right)dy\\
&\leq C(n,p)M^{p}\varepsilon^{p\alpha}\int_{B_{0}(1)}|y|^{p\alpha}\varphi(y)^{p}dy\\
&\leq C(n,p)M^{p}\varepsilon^{p\alpha},
\end{split}
\end{equation*}
which implies
\begin{equation}\label{mollifier lemma equation1}
\|v_{\varepsilon}-v\|_{L^{p}(B_{x_{0}}(r))}\leq C(n,p)M\varepsilon^{\alpha}.
\end{equation}
We then compute
\begin{equation}\label{mollifier lemma equation2}
\begin{split}
\|Dv_{\varepsilon}\|_{L^{p}(B_{x_{0}}(r))}^{p}&=\int_{B_{x_{0}}(r)}\left|\varepsilon^{-1}\int_{B_{0}(1)}\left(v(x+\varepsilon y)-v(x)\right)D\varphi(y)dy\right|^{p}dx\\
&\leq C(n,p)\varepsilon^{-p}\int_{B_{0}(1)}\left(\int_{B_{x_{0}}(r)}\left|v(x+\varepsilon y)-v(x)\right|^{p}|D\varphi(y)|^{p}dx\right)dy\\
&\leq C(n,p)M^{p}\varepsilon^{p(\alpha-1)}\int_{B_{0}(1)}|y|^{p\alpha}|D\varphi(y)|^{p}dy\\
&\leq C(n,p)M^{p}\varepsilon^{p(\alpha-1)}.
\end{split}
\end{equation}
Combining (\ref{mollifier lemma equation2}) and the Poincar\'{e} inequality, we can get
\begin{equation}\label{mollifier lemma equation3}
\begin{split}
\|v_{\varepsilon}-(v_{\varepsilon})_{x_{0},r}\|_{L^{p}(B_{x_{0}}(r))}&\leq C(n,p)r\|Dv_{\varepsilon}\|_{L^{p}(B_{x_{0}}(r))}\\
&\leq C(n,p)M\varepsilon^{\alpha-1}r.
\end{split}
\end{equation}
It is clear that
\begin{equation}\label{mollifier lemma equation4}
\begin{split}
\|(v_{\varepsilon})_{x_{0},r}-v_{x_{0},r}\|_{L^{p}(B_{x_{0}}(r))}&\leq C(n)r^{\frac{n}{p}-n}\int_{B_{x_{0}}(r)}|v_{\varepsilon}(x)-v(x)|dx\\
&\leq C(n)\|v_{\varepsilon}-v\|_{L^{p}(B_{x_{0}}(r))}\\
&\leq C(n,p)M\varepsilon^{\alpha},
\end{split}
\end{equation}
where we use (\ref{mollifier lemma equation1}) in the last line. Combining (\ref{mollifier lemma equation1}), (\ref{mollifier lemma equation3}) and (\ref{mollifier lemma equation4}), we obtain
\begin{equation*}
\begin{split}
&~\|v-v_{x_{0},r}\|_{L^{p}(B_{x_{0}}(r))}\\
\leq&~ \|v-v_{\varepsilon}\|_{L^{p}(B_{x_{0}}(r))}+\|v_{\varepsilon}-(v_{\varepsilon})_{x_{0},r}\|_{L^{p}(B_{x_{0}}(r))}+\|(v_{\varepsilon})_{x_{0},r}-v_{x_{0},r}\|_{L^{p}(B_{x_{0}}(r))}\\
\leq&~C(n,p)M\varepsilon^{\alpha-1}(\varepsilon+r).
\end{split}
\end{equation*}
Since $B_{x_{0}}(r)\subset\Omega'$, we have $r\leq\frac{diam(\Omega)}{2}$, which implies
\begin{equation*}
\frac{d_{0}r}{diam(\Omega)}\leq\frac{d_{0}}{2}<d_{0}.
\end{equation*}
Hence, taking $\varepsilon=\min(r,\frac{d_{0}r}{diam(\Omega)})\in(0,d_{0})$, it then follows that
\begin{equation}\label{mollifier lemma equation5}
\|v-v_{x_{0},r}\|_{L^{p}(B_{x_{0}}(r))}\leq C(n,\alpha,p,d_{0},diam(\Omega))Mr^{\alpha}.
\end{equation}
By (\ref{mollifier lemma equation5}) and the H\"{o}lder inequality, we can get
\begin{equation*}
\begin{split}
\int_{B_{x_{0}}(r)}|v(x)-v_{x_{0},r}|^{2}dx &\leq C(n)r^{n-\frac{2n}{p}}\|v-v_{x_{0},r}\|_{L^{p}(B_{x_{0}}(r))}^{2}\\
&\leq C(n,\alpha,p,d_{0},diam(\Omega))M^{2}r^{n+2(\alpha-\frac{n}{p})}.
\end{split}
\end{equation*}
\end{proof}

The next lemma is due to S. Campanato, which characterizes H\"{o}lder continuous functions by the growth of their local integrals.
\begin{lemma}\label{local integrals lemma}
Suppose $v\in L^{1}(\Omega)$ satisfies
\begin{equation}\label{local integrals lemma equation1}
\int_{B_{q}(r)}|v(x)-v_{q,r}|^{2}dx\leq N^{2}r^{n+2\alpha},
\end{equation}
for any $B_{q}(r)\subset\Omega$, for some $\alpha\in(0,1)$. Then $v\in C^{\alpha}(\Omega)$ and we have the following estimate
\begin{equation*}
\|v\|_{L^{\infty}(\Omega')}\leq C(n,\alpha)\left(Nd_{0}^{\alpha}+d_{0}^{-n}\|v\|_{L^{1}(\Omega)}\right)
\end{equation*}
and
\begin{equation*}
[v]_{C^{\alpha}(\Omega')}\leq C(n,\alpha)\left(N+d_{0}^{-n-\alpha}\|v\|_{L^{1}(\Omega)}\right),
\end{equation*}
for any $\Omega'\subset\subset\Omega$, where $d_{0}=dist(\Omega',\partial\Omega)$.
\end{lemma}

\begin{proof}
For any $q\in\Omega'$ and $0<r_{1}<r_{2}\leq d_{0}$, it is clear that
\begin{equation*}
\begin{split}
|v_{q,r_{1}}-v_{q,r_{2}}|^{2}&\leq\frac{C(n)}{r_{1}^{n}}\int_{B_{q}(r_{1})}|v(x)-v_{q,r_{1}}|^{2}dx+\frac{C(n)}{r_{1}^{n}}\int_{B_{q}(r_{2})}|v(x)-v_{q,r_{2}}|^{2}dx\\
&\leq C(n)N^{2}r_{1}^{-n}(r_{1}^{n+2\alpha}+r_{2}^{n+2\alpha}),
\end{split}
\end{equation*}
where we use (\ref{local integrals lemma equation1}) in the second line. For any $0<r\leq d_{0}$, taking $r_{1}=\frac{r}{2}$ and $r_{2}=r$, we obtain
\begin{equation*}
|v_{q,r}-v_{q,\frac{r}{2}}|\leq C(n,\alpha)Nr^{\alpha},
\end{equation*}
which implies $\{v_{q,\frac{r}{2^{k}}}\}$ is a Cauchy sequence. It then follows that
\begin{equation}\label{local integrals lemma equation2}
|v_{q,r}-\lim_{k\rightarrow\infty}v_{q,\frac{r}{2^{k}}}|\leq C(n,\alpha)Nr^{\alpha}.
\end{equation}
By the Lebesgue theorem, we have
\begin{equation}\label{local integrals lemma equation3}
\lim_{k\rightarrow\infty}v_{q,\frac{r}{2^{k}}}=v(q)~~~a.e.\Omega'.
\end{equation}
Combining (\ref{local integrals lemma equation2}) and (\ref{local integrals lemma equation3}), we can get that $\{v_{q,r}\}$ converges uniformly to $v(q)$ in $\Omega'$. Since $q\mapsto v_{q,r}$ is continuous for any $r>0$, $v(q)$ is also continuous. And we obtain
\begin{equation}\label{local integrals lemma equation4}
|v_{q,r}-v(q)|\leq C(n,\alpha)Nr^{\alpha},
\end{equation}
for any $q\in\Omega'$ and $0<r\leq d_{0}$. It then follows that
\begin{equation}\label{local integrals lemma equation5}
\begin{split}
\|v\|_{L^{\infty}(\Omega')}&\leq C(n,\alpha)Nd_{0}^{\alpha}+\sup_{q\in\Omega'}|v_{q,d_{0}}|\\
&\leq C(n,\alpha)\left(Nd_{0}^{\alpha}+d_{0}^{-n}\|v\|_{L^{1}(\Omega)}\right).
\end{split}
\end{equation}
We now take any $y_{1},y_{2}\in\Omega'$. If $d=|y_{1}-y_{2}|<\frac{d_{0}}{2}$, then we have
\begin{equation}\label{local integrals lemma equation6}
\begin{split}
|v_{y_{1},2d}-v_{y_{2},2d}|^{2}&\leq\frac{2}{|B_{y_{1}}(d)|}\left(\int_{B_{y_{1}}(2d)}|v(x)-v_{y_{1},2d}|^{2}dx+\int_{B_{y_{2}}(2d)}|v(x)-v_{y_{2},2d}|^{2}dx\right)\\
&\leq C(n)N^{2}d^{2\alpha}.
\end{split}
\end{equation}
Combining (\ref{local integrals lemma equation4}) and (\ref{local integrals lemma equation6}), we can obtain
\begin{equation}\label{local integrals lemma equation7}
\begin{split}
|v(y_{1})-v(y_{2})|&\leq|v(y_{1})-v_{y_{1},2d}|+|v_{y_{1},2d}-v_{y_{2},2d}|+|v(y_{2})-v_{y_{2},2d}|\\
&\leq C(n,\alpha)N|y_{1}-y_{2}|^{\alpha}.
\end{split}
\end{equation}
If $|y_{1}-y_{2}|\geq\frac{d_{0}}{2}$, then we can get
\begin{equation*}
|v(y_{1})-v(y_{2})|\leq 2\|v\|_{L^{\infty}(\Omega')}2^{\alpha}d_{0}^{-\alpha}|y_{1}-y_{2}|^{\alpha}.
\end{equation*}
By (\ref{local integrals lemma equation5}), it is clear that
\begin{equation}\label{local integrals lemma equation8}
|v(y_{1})-v(y_{2})|\leq C(n,\alpha)\left(N+d_{0}^{-n-\alpha}\|v\|_{L^{1}(\Omega)}\right)|y_{1}-y_{2}|^{\alpha}.
\end{equation}
Therefore, combining (\ref{local integrals lemma equation5}), (\ref{local integrals lemma equation7}) and (\ref{local integrals lemma equation8}), we complete the proof.
\end{proof}

~~~~~~
\\

\noindent{Jianchun Chu}\\
School of Mathematical Sciences, Peking University\\
Yiheyuan Road 5, Beijing, 100871, China\\
Email:{chujianchun@pku.edu.cn}
\end{document}